\documentclass[12pt]{amsart}
\usepackage{}

\usepackage{amsmath}
\usepackage{amsfonts}
\usepackage{amssymb}
\usepackage[all]{xy}           

\usepackage{bbding}
\usepackage{txfonts}
\usepackage{amscd}

\usepackage[shortlabels]{enumitem}
\usepackage{ifpdf}
\ifpdf
  \usepackage[colorlinks,final,backref=page,hyperindex]{hyperref}
\else
  \usepackage[colorlinks,final,backref=page,hyperindex,hypertex]{hyperref}
\fi
\usepackage{tikz}
\usepackage[active]{srcltx}

\makeatletter

\newtheorem{thm}{Theorem}[section]
\newtheorem{lem}[thm]{Lemma}
\newtheorem{cor}[thm]{Corollary}
\newtheorem{pro}[thm]{Proposition}

\newtheorem{defi}[thm]{Definition}

\newcommand {\emptycomment}[1]{}

\newcommand{\be }{\begin{equation}}
\newcommand{\ee }{\end{equation}}

\newcommand{\g}{\mathfrak g}

\newcommand{\huaC}{{\mathcal{C}}}

\newcommand{\Id}{\rm{Id}}

\newcommand{\br}[1]{   [ \cdot,    \cdot  ]   }

\newcommand{\dM}{\mathrm{d}}

\newcommand{\Hom}{\mathrm{Hom}}

\newcommand{\gl}{\mathfrak {gl}}

\begin{document}

\title[]{Deformations and abelian extensions on anti-pre-Lie algebras }

\author{Shanshan Liu}
\address{School of Mathematics and Statistics, Northeast Normal University, Changchun 130024,  China}
\email{shanshanmath@163.com}
\author{Zhao Chen}
\address{School of Mathematics and Physics, Guangxi Minzu University, Nanning 530007,  China}
\email{czhao0101@163.com}

\author{Liangyun Chen*}
\address{School of Mathematics and Statistics, Northeast Normal University, Changchun 130024, China}
\email{ chenly640@nenu.edu.cn}


\begin{abstract}
In this paper, we introduce the representation of anti-pre-Lie algebras and give the second cohomology group of anti-pre-Lie algebras. As applications, first, we study linear deformations of anti-pre-Lie algebras. The notion of a Nijenhuis operator on an anti-pre-Lie algebra is introduced which can generate a trivial linear deformation of an anti-pre-Lie algebra. Then, we study formal deformations of anti-pre-Lie algebras. We show that the infinitesimal of a formal deformation is a 2-cocycle with the coefficients in the regular representation and depends only on its cohomology class. Moreover, if the second cohomology group $H^2(A;A)$ is trivial, then the anti-pre-Lie algebra is rigid. Finally, we introduce the notion of abelian extensions. We show that abelian extensions are classified by the second cohomology group $H^2(A;V)$.

\end{abstract}

\renewcommand{\thefootnote}{\fnsymbol{footnote}}
\footnote[0]{* Corresponding author}

\maketitle
\vspace{-5mm}
\tableofcontents

\allowdisplaybreaks


\section{Introduction}
The notion  of a pre-Lie algebra (also called  left-symmetric algebras,
quasi-associative algebras, Vinberg algebras and so on) has been introduced independently  by M. Gerstenhaber in  deformation theory  of rings and algebras \cite{Gerstenhaber63}. Pre-Lie algebra arose from the study of affine manifolds and affine structures on Lie group \cite{JLK}, homogeneous convex  cones \cite{Vinberg}. Its defining identity is weaker than associativity. This algebraic structure describes some properties of cochains space in Hochschild cohomology of an associative algebra, rooted trees and vector fields on affine spaces.
 Moreover, it is playing an increasing role in algebra, geometry and physics due to their applications in nonassociative algebras, combinatorics,  numerical Analysis and quantum field theory, see also \cite{Pre-lie algebra in geometry,Bai,Bai2,CK}. There is a close relationship between pre-Lie algebras and Lie algebras: a pre-Lie algebra $(A,\cdot)$ gives rise to a Lie algebra $(A,[\cdot,\cdot]_C)$ via the commutator bracket, which is called the sub-adjacent Lie algebra and denoted by $A^C$. Furthermore, the map $L:A\longrightarrow\gl(A)$, defined by $L_xy=x\cdot y$ for all $x,y\in A$, gives rise to a representation of the sub-adjacent Lie algebra $A^C$ on $A$. 

 A pre-Lie algebra can be induced from a symplectic form on a Lie algebra \cite{BYC}. Similarly, a new kind of algebraic structures, called anti-pre-Lie algebra, has been introduced by G. Liu and C. Bai on \cite{GL}, as the underlying algebraic structures of nondegenerate commutative $2$-cocycles on Lie algebras \cite{AD}. Anti-pre-Lie algebras have some properties which are analogue to the pre-Lie algebras. In fact, an anti-pre-Lie algebra also gives rise to a Lie algebra via the commutator bracket. Furthermore, the negative left multiplication operator gives rise to a representation of this Lie algebra. Consequently there are the constructions of nondegenerate commutative $2$-cocycles and symplectic forms on semi-direct product Lie algebras from anti-pre-Lie algebras and pre-Lie algebras \cite{BAK} respectively. The analogues also appear in the constructions of examples from linear functions and symmetric bilinear forms \cite{Bai1,SIS}. On the other hand, there is a obvious difference between the anti-pre-Lie algebras and the pre-Lie algebras: over the field of characteristic zero, the sub-adjacent Lie algebras of the former can be simple, whereas there is not a compatible pre-Lie algebra structure on a simple Lie algebra \cite{AM}.

Representations and cohomology theories of various kinds of algebras have been developed with a great success. The representation theory of an algebraic
object is very important since it reveals some of its profound structures hidden underneath.  Furthermore, the cohomology theories of an algebraic
object  occupy a  central position since they can give invariants, e.g. they can control deformations and extension problems. Deformation of rings and algebras have been studied by M. Gerstenhaber on \cite{Gerstenhaber1,Gerstenhaber2,Gerstenhaber3,Gerstenhaber4}.

The purpose of this paper is to give a systematic study of a cohomology of an anti-pre-Lie algebra and its application. In Section \ref{sec:coh}, first we recall the notion of anti-pre-Lie algebras and give representations and dual-representations of anti-pre-Lie algebras,  then we provide our main result defining the second cohomology group of an anti-pre-Lie algebra with coefficients in a given representation. In Section \ref{sec:linear}, we study linear deformations of anti-pre-Lie algebras using the cohomology defined in the previous section, and introduce the notion of a Nijenhuis operator on an anti-pre-Lie algebra. We show that a Nijenhuis operator gives rise to a trivial deformation. We study the relation between linear deformations of an anti-pre-Lie algebra and linear deformations of its sub-adjacent Lie algebra. In Section \ref{sec:def}, we study one parameter formal deformations of an anti-pre-Lie algebra using formal power series. We show that the infinitesimal of a formal deformation is a 2-cocycle and depends only on its cohomology class. Moreover, the second cohomology group of anti-pre-Lie algebras can control formal deformations of anti-pre-Lie algebras. In section \ref{sec:ext}, we deal with abelian extensions of anti-pre-Lie algebras. We show that the second cohomology group classifies abelian extensions of an anti-pre-Lie algebra by a given representation.

\vspace{2mm}
\noindent

\section{Representations and second cohomology groups of anti-pre-Lie algebras}\label{sec:coh}
In this section, first we recall the notion of anti-pre-Lie algebras. Then we define representations and dual representations of anti-pre-Lie algebras. Finally, we introduce the second cohomology group of anti-pre-Lie algebras, which will be used to classify infinitesimal deformations and abelian extensions of anti-pre-Lie algebras.

\begin{defi}{\rm(\cite{GL})}
Let A be a vector space with a bilinear map $\cdot:A\otimes A\longrightarrow A$. Then $(A,\cdot)$ is called an {\bf anti-pre-Lie algebra} if for all $x,y,z\in A$, the following equations are satisfied:
\begin{eqnarray}
\label{anti-pre-lie-1}x\cdot(y\cdot z)-y\cdot(x\cdot z)=[y,x]\cdot z,\\
\label{anti-pre-lie-2}[x,y]\cdot z+[y,z]\cdot x+[z,x]\cdot y=0,
\end{eqnarray}
where
\begin{equation}
[x,y]=x\cdot y-y\cdot x.
\end{equation}
  \end{defi}
Let $(A,\cdot)$ be an anti-pre-Lie algebra. The commutator $[x,y]=x\cdot y-y\cdot x$ gives a Lie algebra $(A,[\cdot,\cdot])$, which is denoted by $A^C$ and called  the {\bf sub-adjacent Lie algebra} of $(A,\cdot)$. $(A,\cdot)$ is called the {\bf compatible anti-pre-Lie algebra} of $A^C$. Moreover, $(A,-L)$ is a representation of the sub-adjacent Lie algebra $A^C$, where $L:A\longrightarrow \gl(A)$ is a linear map defined by $L(x)(y)=x\cdot y$ for all $x,y\in A$.

\begin{defi}
 A {\bf morphism} from an anti-pre-Lie algebra $(A,\cdot)$ to an anti-pre-Lie algebra $(A',\cdot')$ is a linear map $f:A\longrightarrow A'$ such that for all
 $x,y\in A$, the following equation is satisfied:
\begin{equation}
\label{ morphism}f(x\cdot y)=f(x)\cdot' f(y),\hspace{3mm}\forall x,y\in A.
\end{equation}
\end{defi}

 \begin{defi}\label{defi:anti-pre-lie representation}
 A {\bf representation} of an anti-pre-Lie algebra $(A,\cdot)$ on
 a vector space $V$ consist of a pair $(\rho,\mu)$, where
  $\rho,\mu:A\longrightarrow \gl(V)$ is a linear map such that for all
  $x,y\in A$, the following equalities are satisfied:
\begin{eqnarray}
\label{anti-pre-lie-rep-1}\rho(x)\circ\rho(y)-\rho(y)\circ\rho(x)&=&\rho[y,x],\\
\label{anti-pre-lie-rep-2}\mu(x\cdot y)-\rho(x)\circ\mu(y)&=&\mu(y)\circ\rho(x)-\mu(y)\circ\mu(x),\\
\label{anti-pre-lie-rep-3}\mu(y)\circ\mu(x)-\mu(x)\circ\mu(y)+\rho[x,y]&=&\mu(y)\circ\rho(x)-\mu(x)\circ\rho(y).
\end{eqnarray}
  \end{defi}
We denote a representation of an anti-pre-Lie algebra $(A,\cdot)$ by a triple  $(V,\rho,\mu)$. Furthermore, let $L,R:A\longrightarrow \gl(A)$ be linear maps, where $L_xy=x\cdot y, R_xy=y\cdot x$. Then $(A,L,R)$ is also a representation, which is called the {\bf regular representation}.

We define a bilinear operation $\cdot_{A\oplus V}:\otimes^2(A\oplus V)\longrightarrow(A\oplus V)$ by
\begin{equation}
(x+u)\cdot_{A\oplus V} (y+v):=x\cdot y+\rho(x)(v)+\mu(y)(u),\quad \forall x,y \in A,  u,v\in V.
\end{equation}

\begin{pro}\label{semi-direct-product}
 With the above notation, $(A\oplus V,\cdot_{A\oplus V})$ is an anti-pre-Lie algebra, which is denoted by $A\ltimes_{(\rho,\mu)}V$ and called  the {\bf semi-direct product} of the anti-pre-Lie algebra $(A,\cdot)$ and the representation $(V,\rho,\mu)$.
\end{pro}
\begin{proof}
For all $x,y,z\in A, u,v,w\in V$, by \eqref{anti-pre-lie-1}, \eqref{anti-pre-lie-rep-1} and \eqref{anti-pre-lie-rep-2}, we have
\begin{eqnarray*}
&&(x+u)\cdot_{A\oplus V}((y+v)\cdot_{A\oplus V}(z+w))-(y+v)\cdot_{A\oplus V}((x+u)\cdot_{A\oplus V}(z+w))-[y+v,x+u]\cdot_{A\oplus V}(z+w)\\
&=&(x+u)\cdot_{A\oplus V}(y\cdot z+\rho(y)w+\mu(z)v)-(y+v)\cdot_{A\oplus V}(x\cdot z+\rho(x)w+\mu(z)u)\\
&&-(y\cdot x+\rho(y)u+\mu(x)v)\cdot_{A\oplus V}(z+w)+(x\cdot y+\rho(x)v+\mu(y)u)\cdot_{A\oplus V}(z+w)\\
&=&x\cdot(y\cdot z)+\rho(x)\rho(y)w+\rho(x)\mu(z)v+\mu(y\cdot z)u-y\cdot(x\cdot z)-\rho(y)\rho(x)w-\rho(y)\mu(z)u-\mu(x\cdot z)v\\
&&-(y\cdot x)\cdot z-\rho(y\cdot x)w-\mu(z)\rho(y)u-\mu(z)\mu(x)v+(x\cdot y)\cdot z+\rho(x\cdot y)w+\mu(z)\rho(x)v+\mu(z)\mu(y)u\\
&=&0,
\end{eqnarray*}
which implies that equation \eqref{anti-pre-lie-1} holds. Similarly, by \eqref{anti-pre-lie-2}, \eqref{anti-pre-lie-rep-2} and \eqref{anti-pre-lie-rep-3}, we obtain
\begin{equation*}
[x+u,y+v]\cdot_{A\oplus V} (z+w)+[y+v,z+w]\cdot_{A\oplus V} (x+u)+[z+w,x+u]\cdot_{A\oplus V} (y+v)=0.
\end{equation*}
This finishes the proof.
\end{proof}
\begin{pro}
Let $(V,\rho,\mu)$ be a representation of an anti-pre-Lie algebra $(A,\cdot)$. Then $(V,\rho-\mu)$ is a representation of the sub-adjacent Lie algebra $A^C$.
\end{pro}
\begin{proof}
For all $x,y\in A$, by \eqref{anti-pre-lie-rep-1}, \eqref{anti-pre-lie-rep-2} and \eqref{anti-pre-lie-rep-3}, we have
\begin{eqnarray*}
&&[(\rho-\mu)(x),(\rho-\mu)(y)]-(\rho-\mu)[x,y]\\
&=&[\rho(x),\rho(y)]-[\rho(x),\mu(y)]-[\mu(x),\rho(y)]+[\mu(x),\mu(y)]-\rho[x,y]+\mu[x,y]\\
&=&\rho(x)\circ\rho(y)-\rho(y)\circ\rho(x)-\rho(x)\circ\mu(y)+\rho(y)\circ\mu(x)+\mu(x\cdot y)-\mu(y\cdot x)\\
&=&\rho[y,x]+\mu(y)\circ\rho(x)-\mu(y)\circ\mu(x)-\mu(x)\circ\rho(y)+\mu(x)\circ\mu(y)\\
&=&0,
\end{eqnarray*}
which implies that
\begin{equation*}
[(\rho-\mu)(x),(\rho-\mu)(y)]=(\rho-\mu)[x,y].
\end{equation*}
This finishes the proof.
\end{proof}
Let $(V,\rho,\mu)$ be a representation of an anti-pre-Lie algebra $(A,\cdot)$. For all $x\in A,u\in V,\xi\in V^*$, define $\rho^*:A\longrightarrow\gl(V^*)$ and $\mu^*:A\longrightarrow\gl(V^*)$ as usual by
$$\langle \rho^*(x)(\xi),u\rangle=-\langle\xi,\rho(x)(u)\rangle,\quad \langle \mu^*(x)(\xi),u\rangle=-\langle\xi,\mu(x)(u)\rangle.$$
\begin{thm}\label{dual-representation}
Let $(A,\cdot)$ be an anti-pre-Lie algebra and $(V,\rho,\mu)$ a representation. Then $(V^*,\mu^*-\rho^*,\mu^*)$ is a representation of $(A,\cdot)$, which is called the {\bf dual representation} of $(V,\rho,\mu)$.
\end{thm}
\begin{proof}
For all $x,y\in A,\xi\in V^*$ and $u\in V$, by \eqref{anti-pre-lie-rep-1}, \eqref{anti-pre-lie-rep-2} and \eqref{anti-pre-lie-rep-3}, we have
\begin{eqnarray*}
&&\langle((\mu^*-\rho^*)(x)(\mu^*-\rho^*)(y)-(\mu^*-\rho^*)(y)(\mu^*-\rho^*)(x)-(\mu^*-\rho^*)[y,x])(\xi),u\rangle\\
&=&\langle(\mu^*(x)\mu^*(y)-\mu^*(x)\rho^*(y)-\rho^*(x)\mu^*(y)+\rho^*(x)\rho^*(y)-\mu^*(y)\mu^*(x)+\mu^*(y)\rho^*(x)\\
&&+\rho^*(y)\mu^*(x)-\rho^*(y)\rho^*(x)-\mu^*[y,x]+\rho^*[y,x])(\xi),u\rangle\\
&=&\langle\xi,(\mu(y)\mu(x)-\rho(y)\mu(x)-\mu(y)\rho(x)+\rho(y)\rho(x)-\mu(x)\mu(y)+\rho(x)\mu(y)+\mu(x)\rho(y)\\
&&-\rho(x)\rho(y)+\mu[y,x]-\rho[y,x])(u)\rangle\\
&=&\langle\xi,(-\rho(y)\mu(x)+\rho(y)\rho(x)+\rho(x)\mu(y)-\rho(x)\rho(y)+\mu(y\cdot x)-\mu(x \cdot y))(u)\rangle\\
&=&\langle\xi,(\mu(x)\rho(y)-\mu(x)\mu(y)-\mu(y)\rho(x)+\mu(y)\mu(x)+\rho[x,y])(u)\rangle\\
&=&0,
\end{eqnarray*}
which implies that
\begin{equation}\label{dual-rep-1}
(\mu^*-\rho^*)(x)\circ(\mu^*-\rho^*)(y)-(\mu^*-\rho^*)(y)\circ(\mu^*-\rho^*)(x)=(\mu^*-\rho^*)[y,x].
\end{equation}
By \eqref{anti-pre-lie-rep-2}, we have
\begin{eqnarray*}
&&\langle(\mu^*(x\cdot y)-(\mu^*-\rho^*)(x)\mu^*(y)-\mu^*(y)(\mu^*-\rho^*)(x)+\mu^*(y)\mu^*(x))(\xi),u\rangle\\
&=&\langle(\mu^*(x\cdot y)-\mu^*(x)\mu^*(y)+\rho^*(x)\mu^*(y)+\mu^*(y)\rho^*(x))(\xi),u\rangle\\
&=&\langle\xi,(-\mu(x\cdot y)-\mu(y)\mu(x)+\mu(y)\rho(x)+\rho(x)\mu(y))(u)\rangle\\
&=&0,
\end{eqnarray*}
which implies that
\begin{equation}\label{dual-rep-2}
\mu^*(x\cdot y)-(\mu^*-\rho^*)(x)\circ\mu^*(y)=\mu^*(y)\circ(\mu^*-\rho^*)(x)-\mu^*(y)\circ\mu^*(x).
\end{equation}
By \eqref{anti-pre-lie-rep-2} and \eqref{anti-pre-lie-rep-3}, we have
\begin{eqnarray*}
&&\langle(\mu^*(y)\mu^*(x)-\mu^*(x)\mu^*(y)+(\mu^*-\rho^*)[x,y]-\mu^*(y)(\mu^*-\rho^*)(x)+\mu^*(x)(\mu^*-\rho^*)(y))(\xi),u\rangle\\
&=&\langle(\mu^*[x,y]-\rho^*[x,y]+\mu^*(y)\rho^*(x)-\mu^*(x)\rho^*(y))(\xi),u\rangle\\
&=&\langle\xi,(-\mu[x,y]+\rho[x,y]+\rho(x)\mu(y)-\rho(y)\mu(x))(u)\rangle\\
&=&\langle\xi,(-\mu(y)\rho(x)+\mu(y)\mu(x)+\mu(x)\rho(y)-\mu(x)\mu(y)+\rho[x,y])(u)\rangle\\
&=&0,
\end{eqnarray*}
which implies that
\begin{equation}\label{dual-rep-3}
\mu^*(y)\circ\mu^*(x)-\mu^*(x)\circ\mu^*(y)+(\mu^*-\rho^*)[x,y]=\mu^*(y)\circ(\mu^*-\rho^*)(x)-\mu^*(x)\circ(\mu^*-\rho^*)(y).
\end{equation}
By \eqref{dual-rep-1}, \eqref{dual-rep-2} and \eqref{dual-rep-3}, we deduce that $(V^*,\mu^*-\rho^*,\mu^*)$ is a representation of $(A,\cdot)$.
\end{proof}
\begin{cor}\label{dual-2-identity}
Let $(V,\rho,\mu)$ be a representation of an anti-pre-Lie algebra $(A,\cdot)$. Then the dual representation of $(V^*,\mu^*-\rho^*,\mu^*)$ is $(V,\rho,\mu).$
\end{cor}
\begin{proof}
It is straightforward.
\end{proof}

Consider the dual representation of the regular representation, we have
\begin{cor}
 Let  $(A,\cdot)$ be an anti-pre-Lie algebra. Then $(A^*,R^*-L^*,R^*)$ is a representation of $(A,\cdot)$.
\end{cor}
\begin{pro}
Let $(V,\rho,\mu)$ be a representation of an anti-pre-Lie algebra $(A,\cdot)$. Then the following conditions are equivalent:
\begin{itemize}
\item [$\rm(i)$]  $(V,\mu-\rho,\mu)$ is a representation of the anti-pre-Lie algebra $(A,\cdot)$,
\item[$\rm(ii)$] $(V^\ast,\rho^*,\mu^*)$ is a representation of the anti-pre-Lie algebra $(A,\cdot)$,
\item[$\rm(iii)$] $\mu(x\cdot y)+\mu(y\cdot x)=0$, for all $x,y\in A$.
\end{itemize}
\end{pro}
\begin{proof}
By Theorem \ref{dual-representation} and Corollary \ref{dual-2-identity}, we obtain that condition $\rm(i)$ is equivalent to condition $\rm(ii)$. If $(V,\mu-\rho,\mu)$ is a representation of $(A,\cdot)$, by \eqref{anti-pre-lie-rep-2}, for all $x,y\in A$, we have
\begin{eqnarray*}
0&=&\mu(x\cdot y)-(\mu-\rho)(x)\circ\mu(y)-\mu(y)\circ(\mu-\rho)(x)+\mu(y)\circ\mu(x)\\
&=&\mu(x\cdot y)-\mu(x)\circ\mu(y)+\rho(x)\circ\mu(y)+\mu(y)\circ\rho(x)\\
&=&2\mu(x\cdot y)-\mu(x)\circ\mu(y)+\mu(y)\circ\mu(x),
\end{eqnarray*}
which implies that
\begin{equation}\label{equivalent1}
2\mu(x\cdot y)-\mu(x)\circ\mu(y)+\mu(y)\circ\mu(x)=0.
\end{equation}
By \eqref{anti-pre-lie-rep-3}, for all $x,y\in A$, we have
\begin{eqnarray*}
0&=&\mu(y)\circ\mu(x)-\mu(x)\circ\mu(y)+(\mu-\rho)[x,y]-\mu(y)\circ(\mu-\rho)(x)+\mu(x)\circ(\mu-\rho)(y)\\
&=&\mu[x,y]-\rho[x,y]+\mu(y)\circ\rho(x)-\mu(x)\circ\rho(y)\\
&=&\mu[x,y]+\mu(y)\circ\mu(x)-\mu(x)\circ\mu(y),
\end{eqnarray*}
which implies that
\begin{equation}\label{equivalent2}
\mu[x,y]+\mu(y)\circ\mu(x)-\mu(x)\circ\mu(y)=0.
\end{equation}
By \eqref{equivalent1} and \eqref{equivalent2}, we have $\mu(x\cdot y)+\mu(y\cdot x)=0$. The converse part can be proved similarly. We omit details. Thus, we deduce that condition $\rm(i)$ is equivalent to condition $\rm(iii)$.
\end{proof}
Let $(V,\rho,\mu)$ be a representation of an anti-pre-Lie algebra $(A,\cdot)$. The set of $n$-cochains is given by
\begin{equation*}
  C^n(A;V)=\Hom(\wedge^n A,V),\quad
 \forall n\geq 0.
\end{equation*}
Now, we define $1$-coboundary operator and $2$-coboundary operator of $(A,\cdot)$ with respect to the representation $(V,\rho,\mu)$.  For all $f\in C^1(A;V)$ and $x,y \in A$, define $\dM^1:C^1(A;V)\longrightarrow C^2(A;V)$ by
\begin{equation*}
\dM^1(f)(x,y)=\rho(x)f(y)+\mu(y)f(x)-f(x\cdot y).
\end{equation*}
For all $f\in C^2(A;V)$ and $x,y,z \in A$, a $2$-coboundary operator of $(A,\cdot)$ on $V$ consists a pair of maps $(\dM^2_1,\dM^2_2)$, define $\dM^2_i:C^2(A;V)\longrightarrow C^3(A;V)$ by
\begin{eqnarray*}
  \dM^2_1(f)(x,y,z)&=&\rho(x)f(y,z)-\rho(y)f(x,z)-\mu(z)f(y,x)+\mu(z)f(x,y)\\
  &&-f(y,x\cdot z)+f(x,y\cdot z)+f([x,y],z),
\end{eqnarray*}
and
\begin{eqnarray*}
  \dM^2_2(f)(x,y,z)&=&\mu(x)(f(y,z)-f(z,y))+\mu(y)(f(z,x)-f(x,z))+\mu(z)(f(x,y)-f(y,x))\\
  &&+f([x,y],z)+f([y,z],x)+f([z,x],y).
\end{eqnarray*}
We denote the set of closed $2$-cochains by $Z^2(A;V)$ and the set of exact $2$-cochains by $B^2(A;V)$.
\begin{pro}
With the above notations, we have $B^2(A;V)\subset Z^2(A;V)$.
\end{pro}
\begin{proof}
For all $f\in C^1(A;V)$, $\dM^1f\in B^2(A;V)$, by \eqref{anti-pre-lie-1}, \eqref{anti-pre-lie-rep-1} and \eqref{anti-pre-lie-rep-2}, for all $x,y,z\in A$, we have
\begin{eqnarray*}
&&\dM^2_1(\dM^1f)(x,y,z)\\
&=&\rho(x)(\dM^1f)(y,z)-\rho(y)(\dM^1f)(x,z)-\mu(z)(\dM^1f)(y,x)+\mu(z)(\dM^1f)(x,y)\\
&&-(\dM^1f)(y,x\cdot z)+(\dM^1f)(x,y\cdot z)+(\dM^1f)([x,y],z)\\
&=&\rho(x)\rho(y)f(z)+\rho(x)\mu(z)f(y)-\rho(x)f(y\cdot z)-\rho(y)\rho(x)f(z)-\rho(y)\mu(z)f(x)+\rho(y)f(x\cdot z)\\
&&-\mu(z)\rho(y)f(x)-\mu(z)\mu(x)f(y)+\mu(z)f(y\cdot x)+\mu(z)\rho(x)f(y)+\mu(z)\mu(y)f(x)-\mu(z)f(x\cdot y)\\
&&-\rho(y)f(x\cdot z)-\mu(x\cdot z)f(y)+f(y\cdot(x\cdot z))+\rho(x)f(y\cdot z)+\mu(y\cdot z)f(x)-f(x\cdot(y\cdot z))\\
&&+\rho[x,y]f(z)+\mu(z)f([x,y])-f([x,y]\cdot z)\\
&=&0.
\end{eqnarray*}
By \eqref{anti-pre-lie-2} and \eqref{anti-pre-lie-rep-3}, we have
\begin{eqnarray*}
&&\dM^2_2(\dM^1f)(x,y,z)\\
&=&\mu(x)((\dM^1f)(y,z)-(\dM^1f)(z,y))+\mu(y)((\dM^1f)(z,x)-(\dM^1f)(x,z))+\mu(z)((\dM^1f)(x,y)\\
&&-(\dM^1f)(y,x))+(\dM^1f)([x,y],z)+(\dM^1f)([y,z],x)+(\dM^1f)([z,x],y)\\
&=&\mu(x)\rho(y)f(z)+\mu(x)\mu(z)f(y)-\mu(x)f(y\cdot z)-\mu(x)\rho(z)f(y)-\mu(x)\mu(y)f(z)+\mu(x)f(z\cdot y)\\
&&+\mu(y)\rho(z)f(x)+\mu(y)\mu(x)f(z)-\mu(y)f(z\cdot x)-\mu(y)\rho(x)f(z)-\mu(y)\mu(z)f(x)+\mu(y)f(x\cdot z)\\
&&+\mu(z)\rho(x)f(y)+\mu(z)\mu(y)f(x)-\mu(z)f(x\cdot y)-\mu(z)\rho(y)f(x)-\mu(z)\mu(x)f(y)+\mu(z)f(y\cdot x)\\
&&+\rho[x,y]f(z)+\mu(z)f([x,y])-f([x,y]\cdot z)+\rho[y,z]f(x)+\mu(x)f([y,z])-f([y,z]\cdot x)\\
&&+\rho[z,x]f(y)+\mu(y)f([z,x])-f([z,x]\cdot y)\\
&=&0.
\end{eqnarray*}
Thus, we obtain that $B^2(A;V)\subset Z^2(A;V)$.
\end{proof}

We denote by $H^2(A;V)=Z^2(A;V)/B^2(A;V)$ the corresponding cohomology groups of the  anti-pre-Lie algebra $(A,\cdot)$ with the coefficient in the representation $(V,\rho,\mu)$.
\section{Linear deformations of anti-pre-Lie algebras}\label{sec:linear}
In this section, we study linear deformations of anti-pre-Lie algebras using the cohomology defined in the previous section, and introduce the notion of a Nijenhuis operator on an anti-pre-Lie algebra. We show that a Nijenhuis operator gives rise to a trivial deformation.
\begin{defi}
Let $(A,\cdot)$ be an anti-pre-Lie algebra. Consider a $t$-parametrized family of multiplication operations:
\begin{equation}
\nonumber x\cdot_t y=x\cdot y+t\Omega(x,y),\quad \forall x,y \in A,
\end{equation}
 where $\Omega:\otimes^2A\longrightarrow A$ is a linear map. If $(A,\cdot_t)$ is still an anti-pre-Lie algebra for all $t$, we say that $\Omega$ generates a {\bf linear deformation} of the anti-pre-Lie algebra $(A,\cdot)$.
\end{defi}
It is direct to check that $\Omega$ generates a linear deformation of the anti-pre-Lie algebra $(A,\cdot)$ if and only if
 for any $x,y,z\in A$, the following equalities are satisfied:
\begin{eqnarray}
\label{deformation-1}x\cdot\Omega(y,z)+\Omega(x,y\cdot z)-y\cdot\Omega(x,z)-\Omega(y,x\cdot z)&&\\
\nonumber -\Omega(y,x)\cdot z-\Omega(y\cdot x,z)+\Omega(x,y)\cdot z+\Omega(x\cdot y,z)&=&0,\\
\label{deformation-2} \Omega(x,y)\cdot z+\Omega(x\cdot y,z)-\Omega(y,x)\cdot z-\Omega(y\cdot x,z)&&\\
\nonumber+\Omega(y,z)\cdot x+\Omega(y\cdot z,x)-\Omega(z,y)\cdot x-\Omega(z\cdot y,x)&&\\
\nonumber+\Omega(z,x)\cdot y+\Omega(z\cdot x,y)-\Omega(x,z)\cdot y-\Omega(x\cdot z,y)&=&0,\\
\label{deformation-3}\Omega(x,\Omega(y,z))-\Omega(y,\Omega(x,z))-\Omega(\Omega(y,x),z)+\Omega(\Omega(x,y),z)&=&0,\\
\label{deformation-4}\Omega(\Omega(x,y)-\Omega(y,x),z)+\Omega(\Omega(y,z)-\Omega(z,y),x)+\Omega(\Omega(z,x)-\Omega(x,z),y)&=&0.
\end{eqnarray}
Obviously, \eqref{deformation-1} and \eqref{deformation-2} mean that $\Omega$ is a $2$-cocycle of the anti-pre-Lie algebra $(A,\cdot)$ with the coefficient in the regular representation $(A,L,R)$, \eqref{deformation-3} and \eqref{deformation-4} mean that $(A,\Omega)$ is  an anti-pre-Lie algebra.
\begin{defi}
Let $(A,\cdot_t)$ and $(A,\cdot'_t)$ be two linear deformations of the anti-pre-Lie algebra $(A,\cdot)$, where $x\cdot_t y=x\cdot y+t\Omega(x,y)$ and $x\cdot'_t y=x\cdot y+t\Omega'(x,y)$. We call them {\bf equivalent} if there exists $N\in \gl(A)$ such that ${\Id_A}+tN$ is a homomorphism from the anti-pre-Lie algebra $(A,\cdot'_t)$ to the anti-pre-Lie algebra $(A,\cdot_t)$, i.e. for all $x,y\in A$, the following equation hold:
\begin{equation*}
({\Id_A}+tN)(x\cdot'_t y)=({\Id_A}+tN)(x)\cdot_t({\Id_A}+tN)(y).
\end{equation*}
In particular, a linear deformation of the anti-pre-Lie algebra $(A,\cdot)$ is said to be {\bf trivial} if it is equivalent to the anti-pre-Lie algebra $(A,\cdot)$.
\end{defi}
We can deduce that anti-pre-Lie algebra $(A,\cdot_t)$ is a trivial deformation if and only if for all $x,y\in A$, the following equations hold:
\begin{eqnarray}
\label{equi-deformation-1}\Omega(x,y)&=&x\cdot N(y)+N(x)\cdot y-N(x\cdot y),\\
\label{equi-deformation-2}N(\Omega(x,y))&=&N(x)\cdot N(y).
\end{eqnarray}
Note that \eqref{equi-deformation-1} means that $\Omega=\dM^1 N$. By \eqref{equi-deformation-1} and \eqref{equi-deformation-2}, we give the notion of Nijenhuis operator of anti-pre-Lie algebras as follow:
\begin{defi}\label{Nijenhuis-operator}
Let $(A,\cdot)$ be an anti-pre-Lie algebra. A linear operator $N\in \gl(A)$ is called a {\bf Nijenhuis operator} on $(A,\cdot)$ if $N$ satisfies the following equation
\begin{equation}\label{Nijenhuis-operator1}
 N(x)\cdot N(y)=N(x\cdot_N y),\quad \forall x,y\in A,
\end{equation}
where the product $\cdot_N$ is defined by
\begin{equation}\label{Nijenhuis-operator}
x\cdot_N y\triangleq  x\cdot N(y)+N(x)\cdot y-N(x\cdot y).
\end{equation}
\end{defi}

By \eqref{equi-deformation-1} and \eqref{equi-deformation-2}, a trivial linear deformation of an anti-pre-Lie algebra gives rise to a Nijenhuis operator $N$. Conversely, a Nijenhuis operator $N$ can also generate a trivial linear deformation as the following theorem shows.
\begin{thm}\label{Nij-deformation}
Let N be a Nijenhuis operator on the anti-pre-Lie algebra $(A,\cdot)$. Then a linear deformation $(A,\cdot_t)$ of the anti-pre-Lie algebra $(A,\cdot)$ can be obtained by putting
\begin{equation*}
\Omega(x,y)=x\cdot_N y.
\end{equation*}
Furthermore, this linear deformation is trivial.
\end{thm}
\begin{proof}
Obviously, $\dM^2_1\Omega=0$ and $\dM^2_2\Omega=0$. Thus, $\Omega$ is a $2$-cocycle of the anti-pre-Lie algebra $(A,\cdot)$ with the coefficient in the regular representation $(A,L,R)$. For all $x,y,z\in A$, we have
\begin{eqnarray*}&&\Omega(x,\Omega(y,z))-\Omega(y,\Omega(x,z))-\Omega(\Omega(y,x),z)+\Omega(\Omega(x,y),z)\\
&=&x\cdot N(y\cdot N(z))+x\cdot N(N(y)\cdot z)-x\cdot N^2(y\cdot z)+N(x)\cdot (y\cdot N(z))\\
&&+N(x)\cdot (N(y)\cdot z)-N(x)\cdot N(y\cdot z)-N(x\cdot (y\cdot N(z)))-N(x\cdot (N(y)\cdot z))\\
&&+N(x\cdot N(y\cdot z))-y\cdot N(x\cdot N(z))-y\cdot N(N(x)\cdot z)+y\cdot N^2(x\cdot z)\\
&&-N(y)\cdot (x\cdot N(z))-N(y)\cdot (N(x)\cdot z)+N(y)\cdot N(x\cdot z)+N(y\cdot (x\cdot N(z)))\\
&&+N(y\cdot (N(x)\cdot z))-N(y\cdot N(x\cdot z))-(y\cdot N(x))\cdot N(z)-(N(y)\cdot x)\cdot N(z)\\
&&+N(y\cdot x)\cdot N(z)-N(y\cdot N(x))\cdot z-N(N(y)\cdot x)\cdot z+N^2(y\cdot x)\cdot z\\
&&+N((y\cdot N(x))\cdot z)+N((N(y)\cdot x)\cdot z)-N(N(y\cdot x)\cdot z)+(x\cdot N(y))\cdot N(z)\\
&&+(N(x)\cdot y)\cdot N(z)-N(x\cdot y)\cdot N(z)+N(x\cdot N(y))\cdot z+N(N(x)\cdot y)\cdot z\\
&&-N^2(x\cdot y)\cdot z-N((x\cdot N(y))\cdot z)-N((N(x)\cdot y)\cdot z)+N(N(x\cdot y)\cdot z).
\end{eqnarray*}
By \eqref{anti-pre-lie-1}, we have
\begin{eqnarray*}
 &&N(x)\cdot (y\cdot N(z))-y\cdot (N(x)\cdot N(z))-(y\cdot N(x))\cdot N(z)+(N(x)\cdot y)\cdot N(z)=0,\\
 &&N(y)\cdot (x\cdot N(z))-x\cdot (N(y)\cdot N(z))-(x\cdot N(y))\cdot N(z)+(N(y)\cdot x)\cdot N(z)=0,\\
 &&N(x)\cdot (N(y)\cdot z)-N(y)\cdot (N(x)\cdot z)-(N(y)\cdot N(x))\cdot z+(N(x)\cdot N(y))\cdot z =0.
\end{eqnarray*}
By \eqref{anti-pre-lie-1} and \eqref{Nijenhuis-operator1}, we have
\begin{eqnarray*}&&-N(x)\cdot N(y\cdot z)+N(y)\cdot N(x\cdot z)+N(y\cdot x)\cdot N(z)-N(x\cdot y)\cdot N(z)\\
&=&-N(x\cdot N(y\cdot z))-N(N(x)\cdot (y\cdot z))+N(y\cdot N(x\cdot z))+N(N(y)\cdot (x\cdot z))\\
&&+N((y\cdot x)\cdot N(z))+N(N(y\cdot x)\cdot z)-N((x\cdot y)\cdot N(z))-N(N(x\cdot y)\cdot z).
\end{eqnarray*}
Therefore, by \eqref{anti-pre-lie-1}, we have
\begin{eqnarray}\label{linear-anti-pre-1}&&\Omega(x,\Omega(y,z))-\Omega(y,\Omega(x,z))-\Omega(\Omega(y,x),z)+\Omega(\Omega(x,y),z)\\
\nonumber&=&-N(x \cdot (y\cdot N(z)))-N(x\cdot (N(y)\cdot z))+N(y\cdot (x\cdot N(z)))+N(y\cdot(N(x)\cdot z))\\
\nonumber&&+N((y\cdot N(x))\cdot z)+N((N(y)\cdot x)\cdot z)-N((x\cdot N(y))\cdot z)-N((N(x)\cdot y)\cdot z)\\
\nonumber&&-N(N(x)\cdot(y\cdot z))+N(N(y)\cdot(x\cdot z))+N((y\cdot x)\cdot N(z))-N((x\cdot y)\cdot N(z))\\
\nonumber&=&0.
\end{eqnarray}
Similarly, by \eqref{anti-pre-lie-2} and \eqref{Nijenhuis-operator1}, we obtain that
\begin{equation}\label{linear-anti-pre-2}
\Omega(\Omega(x,y)-\Omega(y,x),z)+\Omega(\Omega(y,z)-\Omega(z,y),x)+\Omega(\Omega(z,x)-\Omega(x,z),y)=0.
\end{equation}
By \eqref{linear-anti-pre-1} and \eqref{linear-anti-pre-2}, we obtain that $(A,\Omega)$ is an anti-pre-Lie algebra. Thus, $\Omega$ generated a linear deformation of the anti-pre-Lie algebra $(A,\cdot)$.

It is straightforward to deduce that ${\Id_A}+tN$ is a homomorphism from the anti-pre-Lie algebra $(A,\cdot_t)$ to the anti-pre-Lie algebra $(A,\cdot)$. Thus, the linear deformation is trivial.
\end{proof}
\begin{cor}
Let $N$ be a Nijenhuis operator on the anti-pre-Lie algebra $(A,\cdot)$, then $(A,\cdot_N)$ is an anti-pre-Lie algebra, and $N$ is a homomorphism from $(A,\cdot_N)$ to
$(A,\cdot)$.
\end{cor}
At the end of this section, we recall linear deformations of Lie algebras and Nijenhuis operators on Lie algebras, which give trivial deformations of Lie algebras.
\begin{defi}
Let $(\g,[\cdot,\cdot])$ be a Lie algebra and $\Omega\in \huaC^2(\g;\g)$ a skew-symmetric bilinear operator. Consider a $t$-parameterized family of bilinear operations $$[\cdot,\cdot]_t=[\cdot,\cdot]+t\Omega.$$ If $(\g,[\cdot,\cdot]_t)$ is a Lie algebra for all $t$, we say that $\Omega$ generates a {\bf (one-parameter) linear deformation} of a Lie algebra $(\g,[\cdot,\cdot])$.
    \end{defi}
\begin{defi}
Let $(\g,[\cdot,\cdot])$ be a Lie algebra. A linear operator $N\in \gl(A)$ is called a {\bf Nijenhuis operator} on $(\g,[\cdot,\cdot])$ if we have
\begin{equation}
\label{homo-4}[N(x),N(y)]=N[x,y]_N,\quad \forall x,y\in \g.
\end{equation}
where the bracket $[\cdot,\cdot]_N$ is defined by
\begin{equation}
[x,y]_N\triangleq  [N(x),y]+[x,N(y)]-N[x,y].
\end{equation}
    \end{defi}
\begin{pro}
If $\Omega \in C^2(A;A)$ generates a linear deformation of an anti-pre-Lie algebra $(A,\cdot)$, then $\Omega_C \in \huaC^2(A^C;A)$ defined by$$\Omega_C(x,y)=\Omega(x,y)-\Omega(y,x)$$ generates a linear deformation of the sub-adjacent -Lie algebra $A^C$.
\end{pro}

\begin{proof}
Assume that $\Omega$ generates a linear deformation of an anti-pre-Lie algebra $(A,\cdot)$. Then $(A,\cdot_t)$ is an anti-pre-Lie algebra. Consider its corresponding sub-adjacent Lie algebra $(A,[\cdot,\cdot]_t)$, we have
\begin{eqnarray*}
  [x,y]_t&=&x\cdot_t y-y\cdot_t x\\
  &=&x\cdot y+t\Omega(x,y)-y\cdot x-t\Omega(y,x)\\
  &=&[x,y]_C+t\Omega_C(x,y).
\end{eqnarray*}
Thus, $\Omega_C$ generates a linear deformation of $A^C$.
\end{proof}

\begin{pro}
If $N$ is a Nijenhuis operator on an anti-pre-Lie algebra $(A,\cdot)$, then $N$ is a Nijenhuis operator on the sub-adjacent Lie algebra $A^C$.
\end{pro}
\begin{proof}
 For all $x,y\in A$, we have
\begin{eqnarray*}
  [N(x),N(y)]_C&=&N(x)\cdot N(y)-N(y)\cdot N(x)\\
  &=&N\big(N(x)\cdot y+x\cdot N(y)-N(x\cdot y)-N(y)\cdot x-y\cdot N(x)+N(y\cdot x)\big)\\
  &=&N([N(x),y]_C+[x,N(y)]_C-N[x,y]_C).
\end{eqnarray*}
Thus, $N$ is a Nijenhuis operator on the sub-adjacent Lie algebra $A^C$.
\end{proof}

\section{Formal deformations of anti-pre-Lie algebras}\label{sec:def}
In this section, we study formal deformations of anti-pre-Lie algebras. We show that the infinitesimal of a formal deformation is a 2-cocycle and depends only on its cohomology class. Moreover, if the second cohomology group $H^2(A;A)$ is trivial, then the anti-pre-Lie algebra is rigid.

In the sequel, we will denote the anti-pre-Lie multiplication $\cdot$ by $\omega$
\begin{defi}
Let $(A,\omega)$ be an anti-pre-Lie algebra and $\omega_t=\omega+\sum_{i=1}^{+\infty} \omega_it^i:A[[t]]\otimes A[[t]]\longrightarrow A[[t]]$ a $\mathbb K[[t]]$-bilinear map, where $\omega_i:A\otimes A\longrightarrow A$ is a linear map. If $(A[[t]],\omega_t)$ is still an anti-pre-Lie algebra, we say that $\{\omega_i\}_{i\geq1}$ generates a {\bf $1$-parameter formal deformation} of an anti-pre-Lie algebra $(A,\omega)$.
\end{defi}

If $\{\omega_i\}_{i\geq1}$ generates a $1$-parameter formal deformation of an anti-pre-Lie algebra $(A,\omega)$, for all $x,y,z\in A$ and $n=1,2,\dots$,  we have
\begin{equation}\label{cocycle-1}
\sum_{i+j=n\atop i,j\geq 0}\omega_i(x,\omega_j(y,z))-\omega_i(y,\omega_j(x,z))-\omega_i(\omega_j(y,x),z)+\omega_i(\omega_j(x,y),z)=0.
\end{equation}
Moreover, we have
\begin{eqnarray}\label{n-sum-cocycle-1}
&&\sum_{i+j=n\atop 0<i,j\leq n-1 }\omega_i(x,\omega_j(y,z))-\omega_i(y,\omega_j(x,z))-\omega_i(\omega_j(y,x),z)+\omega_i(\omega_j(x,y),z)\\
\nonumber&=&-\dM^2_1\omega_n(x,y,z).
\end{eqnarray}
For all $x,y,z\in A$ and $n=1,2,\dots$,  we have
\begin{equation}\label{cocycle-2}
\sum_{i+j=n\atop i,j\geq 0}\omega_i(\omega_j(x,y)-\omega_j(y,x),z)+\omega_i(\omega_j(y,z)-\omega_j(z,y),x)+\omega_i(\omega_j(z,x)-\omega_j(x,z),y)=0.
\end{equation}
Moreover, we have
\begin{eqnarray}\label{n-sum-cocycle-2}
&&\sum_{i+j=n\atop 0<i,j\leq n-1}\omega_i(\omega_j(x,y)-\omega_j(y,x),z)+\omega_i(\omega_j(y,z)-\omega_j(z,y),x)+\omega_i(\omega_j(z,x)-\omega_j(x,z),y)\\
\nonumber&=&-\dM^2_2\omega_n(x,y,z).
\end{eqnarray}

\begin{pro}\label{pro:2-cocycle-regular-rep}
Let $\omega_t=\omega+\sum_{i=1}^{+\infty} \omega_it^i$ be a $1$-parameter formal deformation of an anti-pre-Lie algebra $(A,\omega)$. Then $\omega_1$ is a $2$-cocycle of the anti-pre-Lie algebra $(A,\omega)$ with coefficients in the regular representation.
\end{pro}
\begin{proof}
When $n=1$, for all $x,y,z\in A$, by \eqref{cocycle-1}, we have
\begin{eqnarray*}
0&=&x\cdot\omega_1(y,z)-y\cdot\omega_1(x,z)-\omega_1(y,x)\cdot z+\omega_1(x,y)\cdot z\\
&&+\omega_1(x,y\cdot z)-\omega_1(y,x\cdot z)-\omega_1(y\cdot x,z)+\omega_1(x\cdot y,z)\\
&=&\dM^2_1\omega_1(x,y,z),
\end{eqnarray*}
and by \eqref{cocycle-2}, we have
\begin{eqnarray*}
 0&=&(\omega_1(x,y)-\omega_1(y,x))\cdot z+(\omega_1(y,z)-\omega_1(z,y))\cdot x+(\omega_1(z,x)-\omega_1(x,z))\cdot y\\
&&+\omega_1([x,y],z)+\omega_1([y,z],x)+\omega_1([z,x],y)\\
&=&\dM^2_2\omega_1(x,y,z).
\end{eqnarray*}
Thus, $\omega_1$ is a $2$-cocycle of the anti-pre-Lie algebra $(A,\omega)$ with coefficients in the regular representation.
\end{proof}

\begin{defi}
The $2$-cocycle $\omega_1$ is called the {\bf infinitesimal} of the $1$-parameter formal deformation $(A[[t]],\omega_t)$ of the anti-pre-Lie algebra $(A,\omega)$.
\end{defi}

\begin{defi}
Let $\omega_t'=\omega+\sum_{i=1}^{+\infty} \omega_i't^i$ and $\omega_t=\omega+\sum_{i=1}^{+\infty} \omega_it^i$ be two $1$-parameter formal deformations of an anti-pre-Lie algebra $(A,\omega)$. A {\bf formal isomorphism} from $(A[[t]],\omega_t')$ to $(A[[t]],\omega_t)$ is a power series $\Phi_t=\sum_{i=0}^{+\infty} \varphi_it^i$, where $\varphi_i:A\longrightarrow A$ are linear maps with $\varphi_0={\Id}$, such that
\begin{equation*}
\Phi_t\circ \omega_t'=\omega_t \circ(\Phi_t\otimes \Phi_t).
\end{equation*}
Two $1$-parameter formal deformations $(A[[t]],\omega_t')$ and $(A[[t]],\omega_t)$ are said to be {\bf equivalent} if there exists a
formal isomorphism $\Phi_t=\sum_{i=0}^{+\infty} \varphi_it^i$ from $(A[[t]],\omega_t')$ to $(A[[t]],\omega_t)$.
\end{defi}

\begin{thm}
Let $(A,\omega)$ be an anti-pre-Lie algebra. If two $1$-parameter formal deformations $\omega_t'=\omega+\sum_{i=1}^{+\infty} \omega_i't^i$ and $\omega_t=\omega+\sum_{i=1}^{+\infty} \omega_it^i$ are equivalent, then the infinitesimals $\omega_1'$ and $\omega_1$ are in the same cohomology class of $H^2(A;A)$.
\end{thm}
\begin{proof}
Let $\omega_t'$ and $\omega_t$ be two $1$-parameter formal deformations. By Proposition \ref{pro:2-cocycle-regular-rep}, we have $\omega_1', \omega_1 \in Z^2(A;A)$. Let  $\Phi_t=\sum_{i=0}^{+\infty} \varphi_it^i$ be the formal isomorphism. Then for all $x,y\in A$, we have
\begin{eqnarray*}
\omega_t'(x,y)
&=&\Phi_t^{-1}\circ\omega_t(\Phi_t(x),\Phi_t(y))\\
&=&({\Id}-\varphi_1 t+\dots)\omega_t\big(x+\varphi_1(x)t+\dots,y+\varphi_1(y)t+\dots\big)\\
&=&({\Id}-\varphi_1 t+\dots)\Big(x\cdot y+\big(x\cdot\varphi_1(y)+\varphi_1(x)\cdot y+\omega_1(x,y)\big)t+\dots\Big)\\
&=&x\cdot y+\Big(x\cdot\varphi_1(y)+\varphi_1(x)\cdot y+\omega_1(x,y)-\varphi_1(x\cdot y)\Big)t+\dots.
\end{eqnarray*}
Thus, we have
\begin{eqnarray*}
\omega_1'(x,y)-\omega_1(x,y)&=&x\cdot\varphi_1(y)+\varphi_1(x)\cdot y-\varphi_1(x\cdot y)\\
&=&\dM^1\varphi_1(x,y),
\end{eqnarray*}
which implies that $\omega_1'-\omega_1=\dM^1\varphi_1.$

Thus, we have $\omega_1'-\omega_1\in B^2(A;A)$. This finishes the proof.
\end{proof}

\begin{defi}
A $1$-parameter formal deformation $(A[[t]],\omega_t)$ of an anti-pre-Lie algebra $(A,\omega)$ is said to be {\bf trivial} if it is equivalent to $(A,\omega)$, i.e. there exists $\Phi_t=\sum_{i=0}^{+\infty} \varphi_it^i$, where $\varphi_i:A\longrightarrow A$ are linear maps with $\varphi_0={\Id}$, such that
\begin{equation*}
 \Phi_t\circ \omega_t=\omega \circ(\Phi_t\otimes \Phi_t).
\end{equation*}
\end{defi}
\begin{defi}
Let $(A,\omega)$ be an anti-pre-Lie algebra. If all $1$-parameter formal deformations are trivial, then $(A,\omega)$ is called {\bf rigid}.
\end{defi}
\begin{thm}
Let $(A,\omega)$ be an anti-pre-Lie algebra. If $H^2(A;A)=0$, then $(A,\omega)$ is rigid.
\end{thm}
\begin{proof}
Let $\omega_t=\omega+\sum_{i=1}^{+\infty} \omega_it^i$ be a $1$-parameter formal deformation and assume that $n\geq1$ is the minimal number such that $\omega_n$ is not zero. By \eqref{n-sum-cocycle-1}, \eqref{n-sum-cocycle-2} and $H^2(A;A)=0$, we have $\omega_n\in B^2(A;A)$. Thus, there exists $\varphi_n \in C^1(A;A)$ such that $\omega_n=\dM^1(-\varphi_n)$. Let $\Phi_t={\Id}+\varphi_nt^n$ and define a new formal deformation $\omega_t'$ by $\omega_t'(x,y)=\Phi_t^{-1}\circ\omega_t(\Phi_t(x),\Phi_t(y))$. Then $\omega_t'$ and $\omega_t$ are equivalent. By straightforward computation, for all $x,y\in A$, we have
\begin{eqnarray*}
\omega_t'(x,y)
&=&\Phi_t^{-1}\circ\omega_t(\Phi_t(x),\Phi_t(y))\\
&=&({\Id}-\varphi_n t^n+\dots)\omega_t\big(x+\varphi_n(x)t^n,y+\varphi_n(y)t^n\big)\\
&=&({\Id}-\varphi_n t^n+\dots)\Big(x\cdot y+\big(x\cdot\varphi_n(y)+\varphi_n(x)\cdot y+\omega_n(x,y)\big)t^n+\dots\Big)\\
&=&x\cdot y+\Big(x\cdot\varphi_n(y)+\varphi_n(x)\cdot y+\omega_n(x,y)-\varphi_n(x\cdot y)\Big)t^n+\dots.
\end{eqnarray*}
Thus, we have $\omega_1'=\omega_2'=\dots=\omega_{n-1}'=0$. Moreover, we have
\begin{eqnarray*}
\omega_n'(x,y)&=&x\cdot \varphi_n(y)+\varphi_n(x)\cdot y+\omega_n(x,y)-\varphi_n(x\cdot y)\\
&=&\dM^1\varphi_n(x,y)+\omega_n(x,y)\\
&=&0.
\end{eqnarray*}
Keep repeating the process, we obtain that $(A[[t]],\omega_t)$ is equivalent to $(A,\omega)$. The proof is finished.
\end{proof}

\section{Abelian extensions of anti-pre-Lie algebras}\label{sec:ext}
In this section, we study abelian extensions of anti-pre-Lie algebras using the cohomological approach. We show that abelian extensions are classified by the second cohomology group $H^2(A;V)$.

\begin{defi}
Let $(A,\cdot)$ and $(V,\cdot_V)$ be two anti-pre-Lie algebras. An {\bf  extension} of $(A,\cdot)$ by $(V,\cdot_V)$ is a short exact sequence of anti-pre-Lie algebra:
$$\xymatrix{
  0 \ar[r] &V \ar[r]^{\iota}& \hat{A}\ar[r]^{p}&A\ar[r]&0,              }$$
where $(\hat{A},\cdot_{\hat{A}})$ is an anti-pre-Lie algebra.

It is called an {\bf abelian extension} if  $(V,\cdot_V)$ is an abelian anti-pre-Lie algebra, i.e. for all $u,v\in V, u\cdot_V v=0$.
\end{defi}
\begin{defi}
A {\bf section} of an extension $(\hat{A},\cdot_{\hat{A}})$ of an anti-pre-Lie algebra $(A,\cdot)$ by  $(V,\cdot_V)$ is a linear map $s:A\longrightarrow \hat{A}$ such that $p\circ s=\Id_A$.
\end{defi}
Let $(\hat{A},\cdot_{\hat{A}})$ be an abelian extension of an anti-pre-Lie algebra $(A,\cdot)$ by $V$ and $s:A\longrightarrow \hat{A}$ a section. For all $x,y\in A$, define linear maps $\theta:A\otimes A\longrightarrow V$  by
\begin{equation*}
\theta(x,y)=s(x)\cdot_{\hat{A}}s(y)-s(x\cdot y).
\end{equation*}
And for all $x,y\in A, u\in V$, define $\rho,\mu:A\longrightarrow\gl(V)$ respectively by
\begin{eqnarray*}
\rho(x)(u)&=&s(x)\cdot_{\hat{A}} u,\\
\mu(x)(u)&=&u\cdot_{\hat{A}} s(x).
\end{eqnarray*}
Obviously, $\hat{A}$ is isomorphic to $A\oplus V$ as vector spaces. Transfer the anti-pre-Lie algebra structure on $\hat{A}$ to that on $A\oplus V$, we obtain an anti-pre-Lie algebra $(A\oplus V,\diamond)$, where $\diamond$ is given by
\begin{equation}
  \label{eq:6.1}(x+u)\diamond (y+v)=x\cdot y+\theta(x,y)+\rho(x)(v)+\mu(y)(u),\quad \forall ~x,y\in A, u,v\in V.
\end{equation}

\begin{lem}\label{lem:representation}
With the above notations, $(V,\rho,\mu)$ is a representation of the anti-pre-Lie algebra $(A,\cdot)$.
\end{lem}
\begin{proof}
For all $x,y\in A$, $u\in V$, by \eqref{anti-pre-lie-1}, we have
\begin{eqnarray*}
0&=&x\diamond(y\diamond u)-y\diamond(x\diamond u)-(y\diamond x)\diamond u+(x\diamond y)\diamond u\\
 &=&x\diamond\rho(y)(u)-y\diamond\rho(x)(u)-(y\cdot x+\theta(y,x))\diamond u+(x\cdot y+\theta(x,y))\diamond u\\
&=&\rho(x)\rho(y)(u)-\rho(y)\rho(x)(u)-\rho([y,x])(u),
\end{eqnarray*}
and
\begin{eqnarray*}
 0&=&u\diamond(x\diamond y)-x\diamond(u\diamond y)-(x\diamond u)\diamond y+(u\diamond x)\diamond y\\
&=&u\diamond(x\cdot y+\theta(x,y))-x\diamond \mu(y)(u)-\rho(x)(u)\diamond y+\mu(x)(u)\diamond y\\
&=&\mu(x \cdot y)(u)-\rho(x)\mu(y)(u)-\mu(y)\rho(x)(u)+\mu(y)\mu(x)(u),
\end{eqnarray*}
which implies that
\begin{eqnarray*}
\rho(x)\circ\rho(y)-\rho(y)\circ\rho(x)&=&\rho([y,x]),\\
\mu(x\cdot y)-\rho(x)\circ\mu(y)&=&\mu(y)\circ\rho(x)-\mu(y)\circ\mu(x).
\end{eqnarray*}
For all $x,y\in A$, $u\in V$, by \eqref{anti-pre-lie-2}, we have
\begin{eqnarray*}
 0&=&(x\diamond y-y\diamond x)\diamond u+(y\diamond u-u\diamond y)\diamond x+(u\diamond x-x\diamond u)\diamond y\\
&=&(x\cdot y+\theta(x,y)-y\cdot x-\theta(y,x))\diamond u+(\rho(y)(u)-\mu(y)(u))\diamond x+(\mu(x)(u)-\rho(x)(u))\diamond y\\
&=&\rho([x,y])(u)+\mu(x)\rho(y)(u)-\mu(x)\mu(y)(u)+\mu(y)\mu(x)(u)-\mu(y)\rho(x)(u),
\end{eqnarray*}
which implies that
\begin{equation*}
\mu(y)\circ\mu(x)-\mu(x)\circ\mu(y)+\rho[x,y]=\mu(y)\circ\rho(x)-\mu(x)\circ\rho(y).
\end{equation*}
Thus, $(V,\rho,\mu)$ is a representation of the anti-pre-Lie algebra $(A,\cdot)$.
\end{proof}

\begin{thm}\label{thm:cocycle}
Let $(\hat{A},\cdot_{\hat{A}})$ be an abelian extension of an anti-pre-Lie algebra $(A,\cdot)$ by $V$. Then $\theta$ is a $2$-cocycle of $(A,\cdot)$ with coefficients in the representation $(V,\rho,\mu)$.
\end{thm}
\begin{proof}
For all $x,y,z\in A$, by \eqref{anti-pre-lie-1}, we have
\begin{eqnarray*}
 0&=&x\diamond(y\diamond z)-y\diamond(x\diamond z)-(y\diamond x)\diamond z+(x\diamond y)\diamond z\\
&=&x\diamond(y\cdot z+\theta(y,z))-y\diamond(x\cdot z+\theta(x,z))-(y\cdot x+\theta(y,x))\diamond z+(x\cdot y+\theta(x,y))\diamond z\\
&=&\theta(x,y\cdot z)+\rho(x)\theta(y,z)-\theta(y,x\cdot z)-\rho(y)\theta(x,z)\\
&&-\theta(y\cdot x,z)-\mu(z)\theta(y,x)+\theta(x\cdot y,z)+\mu(z)\theta(x,y)\\
&=&\dM^2_1\theta(x,y,z).
\end{eqnarray*}
By \eqref{anti-pre-lie-2}, we have
\begin{eqnarray*}
0&=&(x\diamond y-y\diamond x)\diamond z+(y\diamond z-z\diamond y)\diamond x+(z\diamond x-x\diamond z)\diamond y\\
&=&(x\cdot y+\theta(x,y)-y\cdot x-\theta(y,x))\diamond z+(y\cdot z+\theta(y,z)-z\cdot y-\theta(z,y))\diamond x\\
&&+(z\cdot x+\theta(z,x)-x\cdot z-\theta(x,z))\diamond y\\
&=&\theta(x\cdot y,z)+\mu(z)\theta(x,y)-\theta(y\cdot x,z)-\mu(z)\theta(y,x)+\theta(y\cdot z,x)+\mu(x)\theta(y,z)-\theta(z\cdot y,x)\\
&&-\mu(x)\theta(z,y)+\theta(z\cdot x,y)+\mu(y)\theta(z,x)-\theta(x\cdot z,y)-\mu(y)\theta(x,z)\\
&=&\dM^2_2\theta(x,y,z).
\end{eqnarray*}
Thus, $\theta$ is a $2$-cocycle of the anti-pre-Lie algebra $(A,\cdot)$ with coefficients in the representation $(V,\rho,\mu)$. The proof is finished.
\end{proof}
\begin{pro}
Let $(\hat{A},\cdot_{\hat{A}})$ be an abelian extension of an anti-pre-Lie algebra $(A,\cdot)$ by $V$. Then two different sections give rise to the same representation of $(A,\cdot)$.
\end{pro}
\begin{proof}
Choosing two different sections $s_1,s_2:A\longrightarrow \hat{A}$, by Lemma \ref{lem:representation}, we obtain two representations $(V,\rho_1,\mu_1)$ and $(V,\rho_2,\mu_2)$. Define $\varphi:A\longrightarrow V$ by $\varphi(x)=s_1(x)-s_2(x)$. Then for all $x\in A$, we have
\begin{eqnarray*}
 \rho_1(x)(u)-\rho_2(x)(u)&=&s_1(x)\cdot_{\hat{A}} u-s_2(x)\cdot_{\hat{A}} u\\
  &=&(\varphi(x)+s_2(x))\cdot_{\hat{A}} u-s_2(x)\cdot_{\hat{A}} u\\
  &=&\varphi(x)\cdot_{\hat{A}} u\\
  &=&0,
\end{eqnarray*}
which implies that $\rho_1=\rho_2$.
Similarly, we have $\mu_1=\mu_2$. This finishes the proof.
\end{proof}
\begin{defi}\label{defi:morphism}
Let $(\hat{A_1},\cdot_{\hat{A_1}})$ and $(\hat{A_2},\cdot_{\hat{A_2}})$ be two abelian extensions of an anti-pre-Lie algebra $(A,\cdot)$ by  $V$. They are said to be {\bf isomorphic} if there exists an anti-pre-Lie algebra isomorphism $\zeta:(\hat{A_1},\cdot_{\hat{A_1}})\longrightarrow (\hat{A_2},\cdot_{\hat{A_2}})$ such that the following diagram is commutative:
$$\xymatrix{
  0 \ar[r] &V\ar @{=}[d]\ar[r]^{\iota_1}& \hat{A_1}\ar[d]_{\zeta}\ar[r]^{p_1}&A\ar @{=}[d]\ar[r]&0\\
     0\ar[r] &V\ar[r]^{ \iota_2} &\hat{A_2}\ar[r]^{p_2} &A\ar[r]&0.              }$$
\end{defi}
\begin{lem}
Let $(\hat{A_1},\cdot_{\hat{A_1}})$ and $(\hat{A_2},\cdot_{\hat{A_2}})$ be two isomorphic abelian extensions of an anti-pre-Lie algebra $(A,\cdot)$ by $V$. Then they give rise to the same representation of $(A,\cdot)$
\end{lem}
\begin{proof}
 Let $s_1:A_1\longrightarrow \hat{A_1}$ and $s_2:A_2\longrightarrow \hat{A_2}$ be two sections of $(\hat{A_1},\cdot_{\hat{A_1}})$ and $(\hat{A_2},\cdot_{\hat{A_2}})$ respectively. By Lemma \ref{lem:representation}, we obtain that $(V,\rho_1,\mu_1)$ and $(V,\rho_2,\mu_2)$ are their representations respectively. Define $s'_1:A_1\longrightarrow \hat{A_1}$ by $s'_1=\zeta^{-1}\circ s_2$. Since $\zeta:(\hat{A_1},\cdot_{\hat{A_1}})\longrightarrow (\hat{A_2},\cdot_{\hat{A_2}})$ is an anti-pre-Lie algebra isomorphism satisfying the commutative diagram in Definition \ref{defi:morphism}, by $p_2\circ \zeta=p_1$, we have
\begin{equation*}
p_1\circ s'_1=p_2\circ \zeta \circ \zeta^{-1}\circ s_2=\Id_A.
\end{equation*}
Thus, we obtain that $s'_1$ is a section of $(\hat{A_1},\cdot_{\hat{A_1}})$. For all $x\in A, u\in V$, we have
\begin{equation*}
\rho_1(x)(u)=s'_1(x) \cdot_{\hat{A_1}}u=(\zeta^{-1}\circ s_2)(x)\cdot_{\hat{A_1}} u=\zeta^{-1}(s_2(x) \cdot_{\hat{A_2}} u)=\rho_2(x)(u),
\end{equation*}
which implies that $\rho_1=\rho_2$.
Similarly, we have $\mu_1=\mu_2$. This finishes the proof.
\end{proof}
In the sequel, we fix a representation $(V,\rho,\mu)$ of an anti-pre-Lie algebra $(A,\cdot)$ and consider abelian extensions that induce the given representation.
\begin{thm}
Abelian extensions of an anti-pre-Lie algebra $(A,\cdot)$ by $V$ are classified by $H^2(A;V)$.
\end{thm}
\begin{proof}
Let $(\hat{A},\cdot_{\hat{A}})$ be an abelian extension of an anti-pre-Lie algebra $(A,\cdot)$ by  $V$. Choosing a section $s:A\longrightarrow \hat{A}$, by Theorem \ref{thm:cocycle}, we obtain that $\theta\in Z^2(A;V)$. Now we show that the cohomological class of $\theta$ does not depend on the choice of sections. In fact, let $s_1$ and $s_2$ be two different sections. Define $\varphi:A\longrightarrow V$ by $\varphi(x)=s_1(x)-s_2(x)$. Then for all $x,y\in A$, we have
\begin{eqnarray*}
 \theta_1(x,y)&=&s_1(x)\cdot_{\hat{A}} s_1(y)-s_1(x\cdot y)\\
  &=&\big(s_2(x)+\varphi(x)\big)\cdot_{\hat{A}} \big(s_2(y)+\varphi(y)\big)-s_2(x\cdot y)-\varphi(x\cdot y)\\
  &=&s_2(x)\cdot_{\hat{A}} s_2(y)+\rho(x)\varphi(y)+\mu(y)\varphi(x)-s_2(x\cdot y)-\varphi(x\cdot y)\\
  &=&\theta_2(x,y)+\dM^1\varphi(x,y),
\end{eqnarray*}
which implies that $\theta_1-\theta_2=\dM^1\varphi$.
Therefore, we obtain that $\theta_1-\theta_2\in B^2(A;V)$, $\theta_1$ and $\theta_2$ are in the same cohomological class.

Now we  prove that isomorphic abelian extensions give rise to the same element in  $H^2(A;V)$. Assume that $(\hat{A_1},\cdot_{\hat{A_1}})$ and $(\hat{A_2},\cdot_{\hat{A_2}})$ are two isomorphic abelian extensions of an anti-pre-Lie algebra $(A,\cdot)$ by $V$, and $\zeta:(\hat{A_1},\cdot_{\hat{A_1}})\longrightarrow (\hat{A_2},\cdot_{\hat{A_2}})$ is an anti-pre-Lie algebra isomorphism satisfying the commutative diagram in Definition \ref{defi:morphism}. Assume that $s_1:A\longrightarrow \hat{A_1}$ is a section of $\hat{A_1}$. By $p_2\circ \zeta=p_1$, we have
\begin{equation*}
p_2\circ (\zeta\circ s_1)=p_1\circ s_1=\Id_A.
\end{equation*}
Thus, we obtain that $\zeta\circ s_1$ is a section of $\hat{A_2}$. Define $s_2=\zeta\circ s_1$. Since $\zeta$ is an isomorphism of anti-pre-Lie algebras and $\zeta\mid_V=\Id_V$, for all $x,y\in A$, we have
\begin{eqnarray*}
 \theta_2(x,y)&=&s_2(x)\cdot_{\hat{A_2}} s_2(y)-s_2(x\cdot y)\\
  &=&(\zeta\circ s_1)(x)\cdot_{\hat{A_2}}(\zeta\circ s_1)(y)-(\zeta\circ s_1)(x\cdot y)\\
  &=&\zeta\big(s_1(x)\cdot_{\hat{A_1}} s_1(y)-s_1(x\cdot y)\big)\\
  &=&\theta_1(x,y).
\end{eqnarray*}
Thus, isomorphic abelian extensions give rise to the same element in $H^2(A;V)$.

Conversely, given two 2-cocycles $\theta_1$ and $\theta_2$, by \eqref{eq:6.1}, we can construct two abelian extensions $(A\oplus V,\diamond_1)$ and $(A\oplus V,\diamond_2)$. If  $\theta_1, \theta_2,\in H^2(A;V)$, then there exists $\varphi:A\longrightarrow V$, such that $\theta_1=\theta_2+\dM^1\varphi$. We define $\zeta:A\oplus V\longrightarrow A\oplus V$ by
\begin{equation*}
\zeta(x+u)=x+u+\varphi(x),\quad \forall ~x\in A, u\in V.
\end{equation*}
For all $x,y\in A, u,v\in V$, by $\theta_1=\theta_2+\dM^1\varphi$, we have
\begin{eqnarray*}&&\zeta\big((x+u)\diamond_1(y+v)\big)- \zeta(x+u)\diamond_2\zeta(y+v)\\
&=&\zeta\big(x\cdot y+\theta_1(x,y)+\rho(x)(v)+\mu(y)(u)\big)-\big(x+u+\varphi(x)\big)\diamond_2 \big(y+v+\varphi(y)\big)\\
&=&\theta_1(x,y)+\varphi(x\cdot y)-\theta_2(x,y)-\rho(x)\varphi(y)-\mu(y)\varphi(x)\\
&=&\theta_1(x,y)-\theta_2(x,y)-\dM^1\varphi(x,y)\\
&=&0,
\end{eqnarray*}
which implies that $\zeta$ is an anti-pre-Lie algebra isomorphism from $(A\oplus V,\diamond_1)$ to $(A\oplus V,\diamond_2)$. Moreover, it is obvious that the diagram in Definition \ref{defi:morphism} is commutative. This finishes the proof.
\end{proof}

\noindent{\bf Acknowledgement:}  This work is  supported by  NSF of Jilin Province (No. YDZJ202201ZYTS589), NNSF of China (Nos. 12271085, 12071405) and the Fundamental Research Funds for the Central Universities.

 \end{document}